\documentclass{amsart}

\usepackage{etex}
\usepackage{mathtools}
\usepackage{amsmath,amssymb,amsthm,amsfonts,mathrsfs}
\usepackage[frame,cmtip,arrow,matrix,line,graph,curve]{xy}
\usepackage{graphpap,color,paralist,pstricks}
\usepackage[mathscr]{eucal}
\usepackage{mathabx}
\usepackage[pdftex,colorlinks,backref=page,citecolor=blue]{hyperref}
\usepackage{tikz}
\usepackage{epic,eepic}
\usepackage{yfonts}
\usepackage{enumerate}
\usepackage[alphabetic]{amsrefs}
\usepackage{extpfeil}

\theoremstyle{definition}
\newtheorem{theorem}{Theorem}
\newtheorem{definition}[theorem]{Definition}
\newtheorem{conjecture}[theorem]{Conjecture}
\newtheorem{lemma}[theorem]{Lemma}
\newtheorem{proposition}[theorem]{Proposition}
\newtheorem{corollary}[theorem]{Corollary}

\newtheorem*{theorem*}{Theorem}

\theoremstyle{remark}

\newtheorem{example}[theorem]{Example}

\newtheorem*{namedtheorem}{Theorem}
\definecolor{darkgreen}{rgb}{0.0, 0.5, 0.0}

\newcommand{\com}[1]{\textcolor{blue}{[$\star$ #1\ $\star$]}}
\newcommand{\avecom}[1]{\textcolor{red}{[$\star$ #1\ $\star$]}}
\newcommand{\jcom}[1]{\textcolor{purple}{[$\star$ #1\ $\star$]}}
\newcommand{\tojunecom}[1]{\textcolor{darkgreen}{[$\star$ #1\ $\star$]}}

\def\S{\mathfrak{S}}
\def\G{\mathfrak{G}}

\def\Z{\mathbb{Z}}
\def\x{{\bf x}}

\DeclareMathOperator{\wt}{wt}
\DeclareMathOperator{\GL}{GL}

\newcommand{\C}{\mathbb{C}}

\newcommand{\R}{\mathbb{R}}

\newcommand{\fD}{\mathfrak{D}}
\newcommand{\fsl}{\mathfrak{sl}}

\newcommand{\cL}{\mathcal{L}}

\title{Logarithmic Concavity of Schur and related polynomials}

\author{June Huh}
\address{School of Mathematics, Institute for Advanced Study, Princeton, NJ.}
\email{junehuh@ias.edu}

\author{Jacob P. Matherne}
\address{School of Mathematics, Institute for Advanced Study, Princeton, NJ.}
\email{matherne@math.ias.edu}

\author{Karola M\'esz\'aros}
\address{Cornell University, Ithaca, NY, and Institute for Advanced Study, Princeton, NJ.}
\email{karola@math.cornell.edu}

\author{Avery St.~Dizier}
\address{Department of Mathematics, Cornell University, Ithaca, NY.}
\email{ajs624@cornell.edu}

\thanks{June Huh received support from NSF Grant DMS-1638352 and the Ellentuck Fund. 
Jacob Matherne received support from NSF Grant DMS-1638352 and the Association of Members of the Institute for Advanced Study.  
Karola M\'esz\'aros received support from NSF Grant DMS-1501059, CAREER NSF Grant DMS-1847284
and  a von Neumann Fellowship funded by the Friends of the Institute for Advanced Study.}

\begin{document}

\maketitle

\begin{abstract} 
We show that  normalized Schur polynomials are strongly log-concave. 
As a consequence, we obtain Okounkov's log-concavity conjecture for Littlewood--Richardson coefficients in the special case of Kostka numbers. 
\end{abstract}

%%%%%%%%%%%%%%%%%%%%%%%%%%%%%%%%%%%%%%%%%%%%%%%%%%%%%%%%%%%%%%%%%%%%%%%%%%%
\section{Introduction}
%%%%%%%%%%%%%%%%%%%%%%%%%%%%%%%%%%%%%%%%%%%%%%%%%%%%%%%%%%%%%%%%%%%%%%%%%%%

Schur polynomials are the characters of finite-dimensional irreducible polynomial representations of the general linear group $\mathrm{GL}_m(\mathbb{C})$.
Combinatorially, the \emph{Schur polynomial} of a partition $\lambda$ in $m$ variables is the generating function
\[
s_\lambda(x_1,\ldots,x_m)= \sum_{\mathrm{T}} \ x^{\mu(\mathrm{T})}, \quad x^{\mu(\mathrm{T})}= x_1^{\mu_1(\mathrm{T})} \cdots x_m^{\mu_m(\mathrm{T})},
\] 
where the sum is over all Young tableaux $\mathrm{T}$ of shape $\lambda$ with entries from $[m]$,
and 
\[
\mu_i(\mathrm{T})=\text{the number of $i$'s among the entries of $\mathrm{T}$}, \ \  \text{for $i=1,\ldots,m$.}
\]
Collecting Young tableaux of the same weight together, we get
\[
s_\lambda(x_1,\ldots,x_m)= \sum_{\mu} K_{\lambda\mu}\hspace{0.3mm} x^{\mu},
\] 
where $K_{\lambda\mu}$ is the \emph{Kostka number} counting  Young tableaux of given shape $\lambda$ and  weight $\mu$ \cite{Kostka}.
Correspondingly, the \emph{Schur module} $\mathrm{V}(\lambda)$,   an irreducible representation of  the general linear group with highest weight $\lambda$, has the weight space decomposition
\[
\mathrm{V}(\lambda)=\bigoplus_\mu \mathrm{V}(\lambda)_\mu \ \ \text{with} \ \  \dim \mathrm{V}(\lambda)_\mu=K_{\lambda\mu}.
\]
Schur polynomials were first studied by Cauchy  \cite{Cauchy}, who defined them as ratios of alternants.
The connection to the representation theory of $\mathrm{GL}_m(\mathbb{C})$ was found by Schur \cite{Schur}.
For a gentle introduction to these remarkable polynomials, and for all undefined terms, we refer to \cite{FultonYoung}.

We prove several log-concavity properties of Schur polynomials.
An operator that turns generating functions into exponential generating functions will play an important role.
This linear operator, denoted $\mathrm{N}$, is defined by the condition
\[
\mathrm{N}(x^\mu)=\frac{x^\mu}{\mu!}=\frac{x_1^{\mu_1}}{\mu_1!} \cdots \frac{x_m^{\mu_m}}{\mu_m!} \ \ \text{for all $\mu \in \mathbb{N}^m$}.
\]
Recall that a \emph{partition} is a weakly decreasing sequence of nonnegative integers.

\begin{theorem}[Continuous]\label{Continuous}
For any partition $\lambda$, the normalized Schur polynomial 
\[
\mathrm{N}(s_{\lambda}(x_1,\ldots,x_m)) = \sum_{\mu} K_{\lambda\mu}  \frac{x^\mu}{\mu!}
\]
is  either identically zero or its logarithm is concave on the positive orthant $\mathbb{R}^m_{>0}$.
\end{theorem}

Let $e_i$ be the $i$-th standard unit vector in $\mathbb{N}^m$. 
For $\mu \in \mathbb{Z}^m$ and distinct $i,j \in [m]$, we set
\[
\mu(i,j)=\mu+e_i-e_j.
\]
We show that the sequence of weight multiplicities of $\mathrm{V}(\lambda)$ we encounter is always log-concave if we walk in the weight diagram along any root direction $e_i-e_j$.

\begin{theorem}[Discrete]\label{Discrete}
For any partition $\lambda$ and any  $\mu \in \mathbb{N}^m$, we have
\[
K_{\lambda\mu}^2 \ge K_{\lambda \mu(i,j)} K_{\lambda\mu(j,i)} \ \ \text{for any $i,j \in [m]$.}
\]
\end{theorem}

For partitions $\nu,\kappa,\lambda$, the  \emph{Littlewood--Richardson coefficient} $c^\nu_{\kappa\lambda}$ is given by the decomposition 
\[
\mathrm{V}(\kappa) \otimes \mathrm{V}(\lambda)\simeq \bigoplus_\nu \mathrm{V}(\nu)^{\oplus \hspace{0.3mm} c^\nu_{\kappa\lambda}}.
\]
When the skew shape $\nu/\kappa$ has at most one box in each column,
$c^\nu_{\kappa\lambda}$ is the Kostka number $\mathrm{K}_{\lambda\mu}$, where $\mu=\nu-\kappa$.\footnote{The equality between the Littlewood--Richardson coefficient and the Kostka number follows from Pieri's formula 
\[
h_{\mu_1}(x_1,\ldots,x_m) \cdots h_{\mu_m}(x_1,\ldots,x_m)=\sum_\lambda K_{\lambda \mu}s_\lambda(x_1,\ldots,x_m),
\]
 where $h_{\mu_i}$ is the $\mu_i$-th complete symmetric function  \cite[Section 6.1]{FultonYoung}.
When  $\nu/\kappa$ has at most one box in each column,
the left-hand side is  the skew Schur function  $s_{\nu/\kappa}$, given by the Littlewood--Richardson rule
\[
s_{\nu/\kappa}(x_1,\ldots,x_m)=\sum_\lambda c^\nu_{\kappa \lambda} s_\lambda(x_1,\ldots,x_m).
\]
}
Conversely, for any partition $\lambda$ and any $\mu$, we have
\[
K_{\lambda\mu}=c^{\nu}_{\kappa \lambda},
\]
where $\nu$ and $\kappa$ are the partitions given by
$\nu_i=\sum_{j=i}^n \mu_j$ and $\kappa_i=\sum_{j=i+1}^n \mu_j$.
Thus Theorem \ref{Discrete} verifies a special case of Okounkov's conjecture that the discrete function 
\[
(\nu,\kappa,\lambda) \longmapsto \log c^\nu_{\kappa\lambda}
\]
 is concave  \cite[Conjecture 1]{OkounkovWhy}.\footnote{
 The conjecture holds in the ``classical limit'' \cite[Section 3]{OkounkovWhy}, but
 the general case is refuted in \cite{CHJ}: 
 \[
 c^{(4^n, 3^{n}, 2^n, 1^n)}_{(3^n,2^n,1^n)(2^n,1^n,1^n)}={n+2 \choose 2}  \ \ \text{and} \ \   c^{(8^n, 6^{n}, 4^n, 2^n)}_{(6^n,4^n,2^n)(4^n,2^n,2^n)}={n+5 \choose 5} \ \ \text{for all $n$}.
 \]
 The same example shows that the log-concavity conjecture for \emph{parabolic Kostka numbers} \cite[Conjecture 6.17]{Kirillov} also fails.
 }

We point out that, for any fixed $\lambda$, the log-concavity of $K_{\lambda\mu}$ along \emph{any} direction is known to hold \emph{asymptotically}.
By \cite{Heckman}, the \emph{Duistermaat--Heckman measure} obtained from the orbit of $\lambda$ under $\mathrm{SU}_m$  is a translate of the weak limit 
\[
\lim_{k \to \infty} \frac{\sum_\mu K_{k\lambda\hspace{0.3mm} \mu}\delta_{\frac{1}{k}\mu}}{\sum_\mu K_{k\lambda\hspace{0.3mm} \mu}},
\]
where  $\delta_{\frac{1}{k}\mu}$ is the point mass at $\frac{1}{k}\mu$.
It follows from \cite{Graham} that, in this case, the density function of the Duistermaat--Heckman measure is log-concave.\footnote{Let $(\mathrm{M},\omega)$ be a symplectic manifold of dimension $2n$ with an action of a torus $T$ and a moment map $\mathrm{M} \to \mathfrak{t}^*$.
The Duistermaat--Heckman measure is the push-forward of the Liouville measure $\int \omega^n$ via the moment map.
In this generality, Karshon shows that the density function need not be log-concave \cite{Karshon}.}
We refer to  \cite[Section 3]{BGR} for an exposition.

In \cite{Lorentzian}, the authors introduce Lorentzian polynomials as a generalization of volume polynomials in algebraic geometry and stable polynomials in optimization theory.
See Section \ref{SectionMain} for a brief introduction.
We show that normalized Schur polynomials are Lorentzian in the sense of \cite{Lorentzian},
and deduce Theorems \ref{Continuous} and \ref{Discrete} from the Lorentzian property.

\begin{theorem}\label{MainTheorem}
The normalized Schur polynomial $\mathrm{N}(s_{\lambda}(x_1,\ldots,x_m))$ is Lorentzian for any  $\lambda$.
\end{theorem}

Using general properties of Lorentzian polynomials \cite[Section 6]{Lorentzian},
Theorem \ref{MainTheorem} can be strengthened as follows.

\begin{corollary}\label{CorollaryProduct}
For any sequence of partitions  $\lambda^1,\ldots,\lambda^\ell$ and  any positive integers $m_1,\ldots,m_\ell$,
\begin{enumerate}[(1)]\itemsep 5pt
\item the normalized product of Schur polynomials $\mathrm{N}(\prod_{k=1}^\ell s_{\lambda^k}(x_1,\ldots,x_{m_k}))$ is Lorentzian, and
\item the product of normalized Schur polynomials $\prod_{k=1}^\ell \mathrm{N}(s_{\lambda^k}(x_1,\ldots,x_{m_k}))$ is Lorentzian.
\end{enumerate}
\end{corollary}

We prove Theorem \ref{MainTheorem} in Section \ref{SectionMain} in a more general context of Schubert polynomials,
but the main idea
 is simple enough to be outlined here.
The \emph{volume polynomial} of   
an irreducible complex projective variety $Y$,
with respect to a sequence of nef divisor classes\footnote{A Cartier divisor on a complete variety $Y$ is \emph{nef} if it intersects every curve in $Y$ nonnegatively.
We refer to  \cite{Lazarsfeld} for a comprehensive introduction.}
$\mathrm{H}=(\mathrm{H}_1,\ldots,\mathrm{H}_m)$,
 is the homogeneous polynomial 
\[
\text{vol}_{Y,\mathrm{H}}(x_1,\ldots,x_m)=\frac{1}{\dim Y!} \int_Y (x_1\mathrm{H}_1+\cdots+x_m\mathrm{H}_m)^{\dim Y}, 
\]
where the intersection product of $Y$ is used to expand the integrand.
Volume polynomials are prototypical examples of Lorentzian polynomials \cite[Section 10]{Lorentzian}.
To show that the normalized Schur polynomial of  $\lambda$ is a volume polynomial, we suppose that the partition $\lambda$ has $m$ parts, and choose a large integer $\ell$ to get a complementary pair of partitions
\[
\lambda=(\lambda_1,\lambda_2, \ldots,\lambda_m) \ \ \text{and} \ \ \kappa=(\ell,\ell,\ldots,\ell)-(\lambda_m,\lambda_{m-1},\ldots, \lambda_1).
\]
The Schur polynomials of the partitions $\lambda$ and $\kappa$ are related by the identity\footnote{The dual of the Schur module $\mathrm{V}(\lambda)$ has highest weight $(-\lambda_m,\ldots,-\lambda_1)$, see \cite[Exercise 15.50]{FultonHarris}.}
\[
s_\kappa(x_1,\ldots,x_m)=x_1^\ell \cdots x_m^\ell s_\lambda(x_1^{-1},\ldots,x_m^{-1}).
\]
Let $X$ be the product of projective spaces $(\mathbb{P}^\ell)^m$, 
and let $Y$ be a subvariety of $X$ whose fundamental class satisfies
\[
[Y]=s_\kappa(\mathrm{H}_1,\ldots,\mathrm{H}_m) \cap [X], \quad \mathrm{H}_i=c_1(\pi_i^* \mathscr{O}(1)),
\]
where $\pi_i$ is the $i$-th projection. 
The volume polynomial of $Y$ with respect to $\mathrm{H}$ is
\begin{align*}
\text{vol}_{Y,\mathrm{H}}(x_1,\ldots,x_m)&=\frac{1}{\dim Y!} \int_Y (x_1\mathrm{H}_1+\cdots+x_n\mathrm{H}_m)^{\dim Y}\\
&=\frac{1}{\dim Y!}  \int_X s_\kappa(\mathrm{H}_1,\ldots,\mathrm{H}_m) (x_1\mathrm{H}_1+\cdots+x_m\mathrm{H}_m)^{\dim Y}=\mathrm{N} (s_\lambda(x_1,\ldots,x_m)).
\end{align*}
Such $Y$ can be constructed from a sequence of generic global sections $\bigoplus_{i=1}^m \pi_i^* \mathscr{O}(1)$ as a degeneracy locus \cite[Example 14.3.2]{FultonIntersection}, completing the argument.

In Section \ref{SectionMain}, we introduce Lorentzian polynomials and prove the main results. 
In Section \ref{SectionConjecture}, we present evidence for the ubiquity of Lorentzian polynomials through a series of results and conjectures.

\noindent
{\bf Acknowledgments.}
We are grateful to Dave Anderson, Alex Fink, Allen Knutson, Thomas Lam, Ricky  Liu,  Alex Postnikov, Pavlo Pylyavskyy, Vic Reiner, Mark Shimozono, and Alex Yong  for fruitful discussions. 
We thank the Institute for Advanced Study for providing an excellent environment for our collaboration. 

%%%%%%%%%%%%%%%%%%%%%%%%%%%%%%%%%%%%%%%%%%%%%%%%%%%%%%%%%%%%%%%%%%%%%%%%%%%
\section{Normalized Schur polynomials are Lorentzian}\label{SectionMain}
%%%%%%%%%%%%%%%%%%%%%%%%%%%%%%%%%%%%%%%%%%%%%%%%%%%%%%%%%%%%%%%%%%%%%%%%%%%

A subset $\mathrm{J} \subseteq \mathbb{Z}^n$ is \emph{$\mathrm{M}$-convex}\footnote{The letter $\mathrm{M}$ stands for \emph{matroids}. 
When $\mathrm{J} \subseteq \mathbb{N}^n$ consists of zero-one vectors, the $\mathrm{M}$-convexity of $\mathrm{J}$ is the symmetric basis exchange property of matroids \cite[Chapter 4]{White}.}  if, for any index $i \in [n]$ and any $\alpha \in \mathrm{J}$ and $\beta \in \mathrm{J}$ whose $i$-th coordinates satisfy $\alpha_i>\beta_i$, there is an index $j \in [n]$ satisfying
\[
\alpha_j < \beta_j \ \ \text{and} \ \ \alpha-e_i+e_j \in \mathrm{J} \ \ \text{and} \ \ \beta-e_j+e_i \in \mathrm{J}.
\]
The notion of $\mathrm{M}$-convexity forms the foundation of discrete convex analysis \cite{Murota}.
The convex hull of an $\mathrm{M}$-convex set  is a \emph{generalized permutohedron} in the sense of \cite{Postnikov},
and conversely, the set of integral points in an integral generalized permutohedron is an $\mathrm{M}$-convex set \cite[Theorem 1.9]{Murota}.

Lorentzian polynomials connect discrete convex analysis with  many log-concavity phenomena in combinatorics.
See \cite{AOGV18,ALOGV18a,ALOGV18b,BES,BH18,Lorentzian,EH19} for recent applications.
Here we briefly summarize the relevant results, and refer to \cite{Lorentzian} for details.
We fix integers $d$ and $e=d-2$.

\begin{definition}\label{LorentzianDefinition}
Let $h(x_1,\ldots,x_n)$ be a degree $d$ homogeneous polynomial.
We say that $h$ is \emph{strictly Lorentzian} if
all the coefficients of $h$ are positive and 
\[
\text{$\frac{\partial}{\partial x_{i_1}} \cdots \frac{\partial}{\partial x_{i_{e}}}h$ has the signature $(+,-,\ldots,-)$ for any  $i_1,\ldots,i_{e} \in [n]$.}
\]
We say that $h$ is \emph{Lorentzian} if it satisfies any one of the following equivalent conditions.
\begin{enumerate}[(1)]\itemsep 5pt
\item All the coefficients of $h$ are nonnegative, the support of $h$ is $\mathrm{M}$-convex,\footnote{The \emph{support} of a polynomial $h(x_1,\ldots,x_n)$ is the set of monomials appearing in $h$, viewed as a subset of $\mathbb{N}^n$.} and
\[
\text{$\frac{\partial}{\partial x_{i_1}} \cdots \frac{\partial}{\partial x_{i_{e}}}h$ has at most one positive eigenvalue  for any  $i_1,\ldots,i_{e} \in [n]$.}
\]
\item All the coefficients of $h$ are nonnegative and, for any $i_1,i_2,\ldots \in [n]$ and any positive  $k$,
\[
\text{the functions $h$ and $\frac{\partial}{\partial x_{i_1}} \cdots \frac{\partial}{\partial x_{i_{k}}}h$ are either identically zero or log-concave on $\mathbb{R}^n_{>0}$.}
\]
\item The polynomial $h$ is a limit of strictly Lorentzian polynomials.
\end{enumerate}
\end{definition}

For example, a bivariate polynomial $\sum_{k=0}^d a_k x_1^k x_2^{d-k}$ with nonnegative coefficients is Lorentzian
if and only if the sequence $a_0,\ldots,a_d$ has no internal zeros\footnote{The sequence $a_0,\ldots,a_d$ has \emph{no internal zeros} if $a_{k_1} a_{k_3} \neq 0 \Longrightarrow a_{k_2}\neq 0$ for all $0 \le k_1 <k_2<k_3\le d$.} 
and 
\[
\frac{a_k^2}{{d \choose k}^2} \ge \frac{a_{k-1}}{{d \choose k-1}} \frac{a_{k+1}}{{d \choose k+1}} \ \ \text{for all $0<k<d$.}
\]
Polynomials satisfying the second condition of Definition \ref{LorentzianDefinition}, introduced by Gurvits in \cite{Gurvits}, are called \emph{strongly log-concave}.
See \cite[Section 5]{Lorentzian} for a proof of the equivalence of the three conditions in Definition \ref{LorentzianDefinition}.

We write $\mathcal{S}_n$ for the group of permutations of $[n]$.
The \emph{Schubert polynomial} $\mathfrak{S}_w(x_1,\ldots,x_n)$ for $w \in \mathcal{S}_n$ can be defined recursively as follows.
\begin{enumerate}[(1)]\itemsep 5pt
\item If $w=w_\circ$ is the longest permutation $n\ n-1 \ \cdots  \ 2\ 1$, then
\[
\mathfrak{S}_{w}(x_1,\ldots,x_n)=x_1^{n-1}x_2^{n-2} \cdots  x_{n-1}^1.
\]
\item If $w(i)>w(i+1)$ for some $i$ and $s_i$ is the adjacent transposition $(i \ i+1)$, then
\[
\mathfrak{S}_{ws_i}(x_1,\ldots,x_n)=\partial_i \mathfrak{S}_{w}(x_1,\ldots,x_n).
\]
\end{enumerate}
The symbol $\partial_i$ stands for the \emph{$i$-th divided difference operator} defined by the formula
\[
\partial_i \mathfrak{S}_{w}=\frac{\mathfrak{S}_{w}- s_i\mathfrak{S}_{w} }{x_i-x_{i+1}},
\]
where $s_i\mathfrak{S}_w$ is the polynomial obtained from $\mathfrak{S}_w$ by interchanging $x_i$ and $x_{i+1}$.
The divided difference operators satisfy the \emph{braid relations},
and it follows that the Schubert polynomials are well-defined \cite[Exercise 15.3]{MillerSturmfels}.
For any $w \in \mathcal{S}_n$, we define
\[
\mathfrak{S}_w^\vee =\mathrm{N}(x_1^{n-1}\cdots x_n^{n-1}\mathfrak{S}_w(x_1^{-1},\ldots,x_n^{-1})).
\]

\begin{theorem}\label{SchubertComplement}
The polynomial $\mathfrak{S}_w^\vee(x_1,\ldots,x_n)$ is Lorentzian for any $w \in \mathcal{S}_n$.
\end{theorem}

We conjecture that $\mathrm{N}(\mathfrak{S}_w(x_1,\ldots,x_n))$ is Lorentzian for any $w \in \mathcal{S}_n$, see Section \ref{SectionSchubert}.

\begin{proof}
Recall that the volume polynomial of   
a projective variety $Y$,
with respect to a sequence of Cartier divisor classes $\mathrm{H}=(\mathrm{H}_1,\ldots,\mathrm{H}_n)$, is the homogeneous polynomial 
\[
\text{vol}_{Y,\mathrm{H}}(x_1,\ldots,x_n)=\frac{1}{\dim Y!} \int_Y (x_1\mathrm{H}_1+\cdots+x_n\mathrm{H}_n)^{\dim Y}.
\]
By  \cite[Theorem 10.1]{Lorentzian}, the volume polynomial is Lorentzian whenever $Y$ is irreducible and $\mathrm{H}_1,\ldots,\mathrm{H}_n$ are nef.
We show that $\mathfrak{S}_w^\vee$ is a volume polynomial for suitable $Y=Y_w$ and $\mathrm{H}$.

Let $X$ be the product of projective spaces $(\mathbb{P}^{n-1})^n$.
We write $x_{i1},x_{i2},\ldots,x_{in}$ for the homogeneous coordinates of the $i$-th projective space,
and write $\pi_i$ for the $i$-th projection. 
We consider the map between the rank $n$ vector bundles
\[
\Psi:\bigoplus_{i=1}^n \mathscr{O}_X \longrightarrow \bigoplus_{j=1}^n \pi_j^*\mathscr{O}(1), \quad \Psi(x)=(x_{ij})_{1 \le i \le n, 1 \le j \le n}.
\]
For $p,q \in [n]$, the induced map
$\bigoplus_{i=1}^p \mathscr{O}_X \rightarrow
 \bigoplus_{j=1}^q \pi_j^*\mathscr{O}(1)
$
will be denoted $\Psi_{p \times q}$.
We set
\[
Y=Y_w\vcentcolon=\Big\{x \in X \mid \text{$\text{rank}\ \Psi_{p\times q}(x) \le \text{rank}\ w_{p\times q}$ for all $p$ and $q$}\Big\},
\]
where $w_{p\times q}$ is the $p \times q$  partial permutation matrix with $ij$-entry $1$ for $w(i)=j$.
The locus $Y$ is defined by all minors of $(x_{ij})_{1 \le i \le p, 1 \le j \le q}$ of size one more than  the rank of $w_{p \times q}$ for all $p$ and $q$.

By \cite[Theorem 8.2]{FultonFlags}, the fundamental class of $Y$ in the Chow group of $X$ is given by
\[
[Y]=\mathfrak{S}_w(\mathrm{H}_1,\ldots,\mathrm{H}_n) \cap [X], \quad \mathrm{H}_i=c_1(\pi_i^* \mathscr{O}(1)).
\]
An alternative proof of the displayed formula, in a more refined setting, was obtained in \cite{KnutsonMiller} through an explicit degeneration of $Y$.
An important point for us is that $Y$ is  irreducible of expected codimension $\text{deg}\ \mathfrak{S}_w$ \cite{FultonFlags}. 
For an elementary proof that the multi-homogeneous ideal defining $Y$ is prime, see  \cite[Section 16.4]{MillerSturmfels}.
The volume polynomial of $Y$ with respect to $\mathrm{H}=(\mathrm{H}_1,\ldots,\mathrm{H}_n)$ is
\begin{align*}
\text{vol}_{Y,\mathrm{H}}(x_1,\ldots,x_n) &=\frac{1}{\dim Y!}\int_{Y} (x_1\mathrm{H}_1 +\cdots+x_n\mathrm{H}_n)^{\dim Y}\\
&=\frac{1}{\dim Y!}\int_X \mathfrak{S}_w(\mathrm{H}_1,\ldots,\mathrm{H}_n) (x_1\mathrm{H}_1 +\cdots+x_n\mathrm{H}_n)^{\dim Y}= \mathfrak{S}_w^\vee(x_1,\ldots,x_n).
\end{align*}
The second equality is the projection formula, and the third equality follows from
\[
\int_X \mathrm{H}^\mu =\left\{\begin{array}{cc} 1 & \text{if $\mu=(n-1,\ldots,n-1)$,} \\ 0& \text{if $\mu\neq (n-1,\ldots,n-1)$.} \end{array}\right.
\]
Now the Lorentzian property of $\mathfrak{S}_w^\vee$ can be deduced from \cite[Theorem 10.1]{Lorentzian}.
\end{proof}

\begin{lemma}\label{Translation}
For any $\mu \in \mathbb{N}^n$ and any polynomial $f=f(x_1,\ldots,x_n)$, 
\[
\text{$\mathrm{N}(f)$ is Lorentzian if and only if $\mathrm{N}(x^\mu f)$ is Lorentzian}. 
\]
\end{lemma}

\begin{proof} 
If a polynomial $g(x_1,\ldots,x_n)$ is Lorentzian, then so is its partial derivative 
\[
\partial^\mu g=\Big(\frac{\partial}{\partial x_1}\Big)^{\mu_1} \cdots \Big(\frac{\partial}{\partial x_n}\Big)^{\mu_n} g(x_1,\ldots,x_n).
\]
Therefore, the ``if'' direction follows from the equality of linear operators
\[
\partial^\mu \circ  \mathrm{N} \circ x^\mu=\mathrm{N}.
\]
The ``only if'' direction is a special case of \cite[Corollary 6.8]{Lorentzian}.
\end{proof}

\begin{proof}[Proof of Theorem \ref{MainTheorem}]
As in the introduction, given a partition $\lambda$ with $m$ parts, 
we choose a large integer $\ell$ and write $\kappa$ for the partition complementary to $\lambda$ in  the 
$m \times \ell$ rectangle. 
Choose another  large integer $n$, 
and let $w$  be the unique element of $\mathcal{S}_n$ satisfying 
\[
\kappa=\big(w(m)-m, \ldots, w(1)-1\big) \ \ \text{and} \ \ w(m)>w(m+1)<w(m+2)<\cdots<w(n).
\]
The element $w$ is the \emph{Grassmannian permutation} in $\mathcal{S}_n$ with the \emph{Lehmer code} 
\[
L(w)=(w(1)-1,\ldots,w(m)-m,0,\ldots,0)=(\kappa_m,\ldots,\kappa_1,0,\ldots,0).
\]
The Schubert polynomial of $w$ satisfies
\[
\mathfrak{S}_w(x_1,\ldots,x_n)=s_\kappa(x_1,\ldots,x_m)=x_1^\ell \cdots x_m^\ell s_\lambda(x_1^{-1},\ldots,x_m^{-1}),
\]
where the first equality is \cite[Proposition 2.6.8]{Manivel} and the second equality is \cite[Exercise 15.50]{FultonHarris}.
By Theorem \ref{SchubertComplement}, we know that the polynomial $\mathfrak{S}^\vee_w$ is Lorentzian, which is equal to
\[
\mathrm{N}(x_1^{n-1}\cdots x_n^{n-1} s_\kappa(x_1^{-1},\ldots,x_m^{-1}))=
\mathrm{N}(x^{\mu} s_\lambda(x_1,\ldots,x_m)) \ \text{for some $\mu \in \mathbb{N}^n$.}
\]
Therefore, by Lemma \ref{Translation}, the Lorentzian property of $\mathfrak{S}_w^\vee$ implies that of $\mathrm{N}(s_\lambda(x_1,\ldots,x_m))$.
\end{proof}

\begin{proof}[Proofs of Theorems \ref{Continuous} and \ref{Discrete}]
Since any nonzero Lorentzian polynomial  is log-concave on the positive orthant,
Theorem \ref{Continuous} follows from Theorem \ref{MainTheorem}.
For Theorem \ref{Discrete}, we may suppose that 
\[
\mu_1+\cdots+\mu_m=\lambda_1+\cdots+\lambda_m \ge 2 \ \ \text{and} \ \  \kappa\vcentcolon=\mu-e_i-e_j \in \mathbb{N}^m.
\]
We consider the quadratic form with at most one positive eigenvalue
\[
\frac{\partial^{\kappa_1}}{\partial x_1^{\kappa_1}} \cdots \frac{\partial^{\kappa_m}}{\partial x_m^{\kappa_m}}\mathrm{N}(s_\lambda(x_1,\ldots,x_m)),
\]
viewed  as an $m \times m$ symmetric matrix.
Its $2 \times 2$ principal submatrix corresponding to $i$ and $j$ is either identically zero or has exactly one positive eigenvalue,
by Cauchy's interlacing theorem.
The nonpositivity of the $2 \times 2$ principal minor gives the conclusion
\[
K_{\lambda \mu}^2 \ge K_{\lambda \mu(i,j)}K_{\lambda \mu(j,i)}. \qedhere
\]
\end{proof}

\begin{proof}[Proof of Corollary \ref{CorollaryProduct}]
The first part follows from Theorem \ref{MainTheorem} and \cite[Corollary 6.8]{Lorentzian}.
The second part follows from Theorem \ref{MainTheorem} and  \cite[Corollary 5.5]{Lorentzian}.
\end{proof}

In general, if $h$ is a Lorentzian polynomial, then its normalization $\mathrm{N}(h)$ is a Lorentzian polynomial \cite[Corollary 6.7]{Lorentzian}.
We record here that Schur polynomials, before the normalization, need not be Lorentzian.

\begin{example}
The Schur polynomial of the partition $\lambda=(2,0)$ in two variables is 
\[
s_\lambda(x_1,x_2)=x_1^2+x_1x_2+x_2^2.
\]
The quadratic form has eigenvalues $\frac{3}{2}$ and $\frac{1}{2}$, and hence $s_\lambda$ is not Lorentzian.
\end{example}

A polynomial $f(x_1,\ldots,x_m)$ is \emph{stable} if 
$f$ has no zeros in the product of $m$ open upper half planes \cite{Wagner}.
Homogeneous stable polynomials with nonnegative coefficients are motivating examples of Lorentzian polynomials \cite[Proposition 2.2]{Lorentzian}.
We record here that normalized Schur polynomials, although Lorentzian, need not be stable.

\begin{example}
The normalized Schur polynomial of  $\lambda=(3,1,1,1,1)$ in five variables is
\[
\mathrm{N}(s_{\lambda}(x_1,\ldots,x_5))=\frac{1}{12}x_1x_2x_3x_4x_5\Big( \sum_{1\leq i< j\leq 5} 3x_ix_j +\sum_{1 \le i\le 5} 2x_i^2\Big).
\]
By \cite[Lemma 2.4]{Wagner}, if $\mathrm{N}(s_{\lambda})$ is stable, then so is its univariate specialization
\[
\mathrm{N}(s_{\lambda})|_{x_2=x_3=x_4=x_5=1}=\frac{1}{6} x_1 \Big(x_1^2+6x_1+13\Big).
\]
However, the displayed cubic has a pair of nonreal zeros, and hence  $\mathrm{N}(s_{\lambda})$ is not stable. 
\end{example}

%%%%%%%%%%%%%%%%%%%%%%%%%%%%%%%%%%%%%%%%%%%%%%%%%%%%%%%%%%%%%%%%%%%%%%%%%%%
\section{Ubiquity of Lorentzian polynomials}\label{SectionConjecture}
%%%%%%%%%%%%%%%%%%%%%%%%%%%%%%%%%%%%%%%%%%%%%%%%%%%%%%%%%%%%%%%%%%%%%%%%%%%

%---------------------------------------------------------------------------
\subsection{Multiplicities of highest weight modules}\label{SectionMultiplicity}
%---------------------------------------------------------------------------

We point to \cite{Humphreys} for background on representation theory of semisimple Lie algebras.
Let $\Lambda$ be the integral weight lattice of the Lie algebra $\mathfrak{sl}_m(\C)$,
 let $\varpi_1,\ldots,\varpi_{m-1}$ be the fundamental weights, and let $\rho$ be the sum of the fundamental weights.
For $\lambda \in \Lambda$, we write $\mathrm{V}(\lambda)$ for the irreducible  $\mathfrak{sl}_m(\C)$-module with highest weight $\lambda$,
and consider its decomposition  into finite-dimensional weight spaces
\[
\mathrm{V}(\lambda)=\bigoplus_\mu \mathrm{V}(\lambda)_\mu.
\]
For $\mu \in \Lambda$ and distinct $i,j \in [m]$, we write
$
\mu(i,j)$ for the element $\mu+e_i-e_j \in \Lambda$.

\begin{conjecture}\label{MultiplicityConjecture}
	For any $\lambda \in \Lambda$  and any $\mu \in \Lambda$, we have
\[
(\dim \mathrm{V}(\lambda)_\mu)^2 \ge \dim \mathrm{V}(\lambda)_{\mu(i,j)} \dim \mathrm{V}(\lambda)_{\mu(j,i)} \ \ \text{for any $i,j \in [m]$.}
\]
\end{conjecture}

When $\lambda$ is dominant, the dimension of the weight space $\mathrm{V}(\lambda)_\mu$ is the Kostka number $K_{\lambda\mu}$,
and Theorem \ref{Discrete} shows that Conjecture \ref{MultiplicityConjecture} holds in this case.
When $\lambda$ is antidominant \cite[Section 4.4]{Humphreys}, $\mathrm{V}(\lambda)$ is the \emph{Verma module} $\mathrm{M}(\lambda)$, the universal highest weight module of  highest weight $\lambda$.
We note that Conjecture \ref{MultiplicityConjecture} holds in this case as well.

\begin{proposition}\label{VermaLogConcave}
For any $\lambda \in \Lambda$  and any $\mu \in \Lambda$, we have
\[
(\dim \mathrm{M}(\lambda)_\mu)^2 \ge \dim \mathrm{M}(\lambda)_{\mu(i,j)} \dim \mathrm{M}(\lambda)_{\mu(j,i)} \ \ \text{for any $i,j \in [m]$.}
\]
\end{proposition}

One may deduce Proposition \ref{VermaLogConcave} from its stronger variant Proposition \ref{VermaLorentzian} below.

\begin{proof}[Alternative proof]
The Poincar\'{e}--Birkhoff--Witt theorem shows that  the dimensions of the weight spaces are given by the Kostant partition function $p$:
 \[
 \dim \mathrm{M}(\lambda)_\mu=p(\mu-\lambda)=\text{number of ways to write $\mu-\lambda$ as a sum of negative roots}.
 \]  
Lidskij's volume formula for flow polytopes shows that all Kostant partition function evaluations are mixed volumes of Minkowski sums of polytopes \cite{BaldoniVergne}.
The Alexandrov--Fenchel inequality for mixed volumes \cite[Section 7.3]{Schneider} yields the desired log-concavity property.
\end{proof}

The  diagram below shows some of the weight multiplicities of the  irreducible  $\mathfrak{sl}_4(\mathbb{C})$-module  with highest weight $-2\varpi_1-3\varpi_2$.
We start from the highlighted vertex $\varpi_1-6\varpi_2-3\varpi_3$ and walk along negative root directions in  the hyperplane spanned by $e_2-e_1$ and $e_3-e_2$.  
In the shown region, 
the sequence of weight multiplicities along any line is  log-concave, as predicted by  Conjecture \ref{MultiplicityConjecture}.
\begin{figure}[h]
\includegraphics[scale=.8]{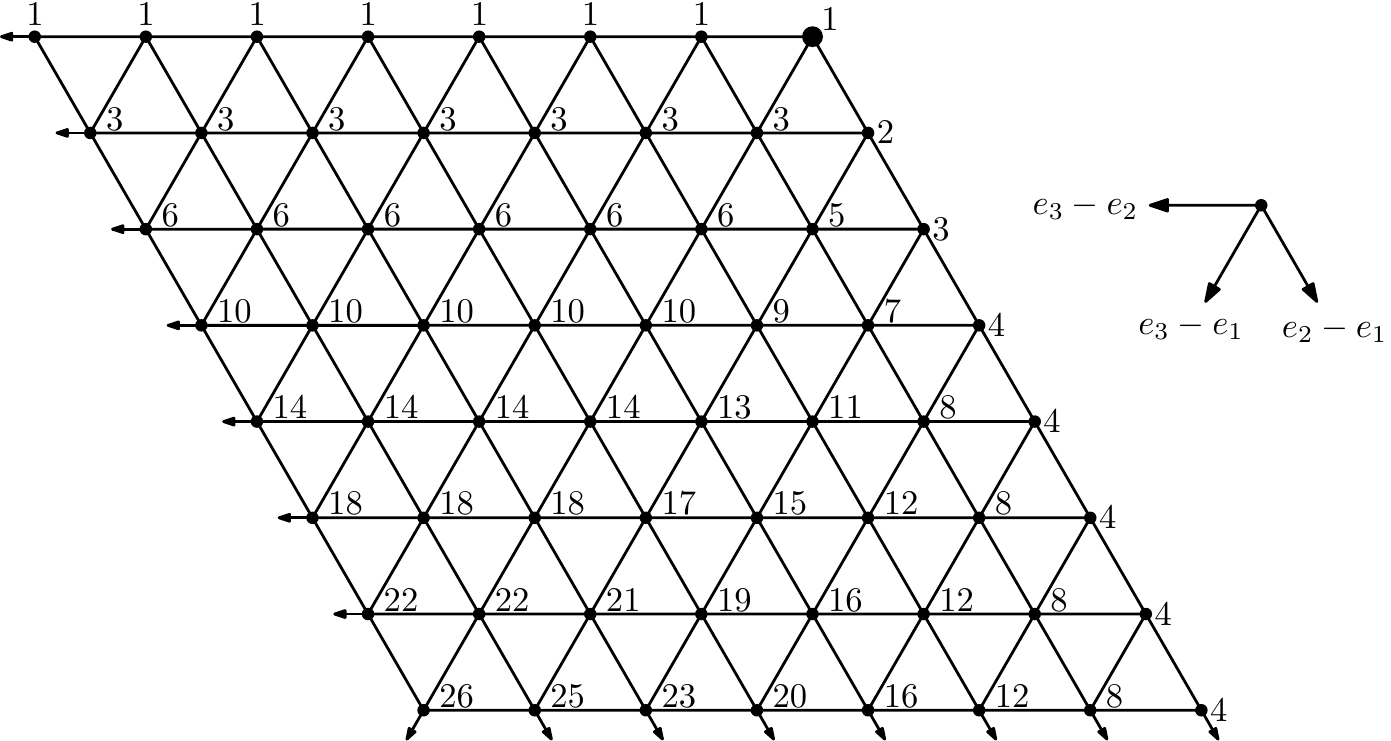}
\end{figure}

We note, however, that a naive analog of Conjecture \ref{MultiplicityConjecture} does not hold for symplectic Lie algebras.
In the weight diagram of the irreducible representation of $\mathfrak{sp}_4(\mathbb{C})$ with highest weight $2\varpi_2$ shown below,
\begin{figure}[h]
		\includegraphics[scale=.73]{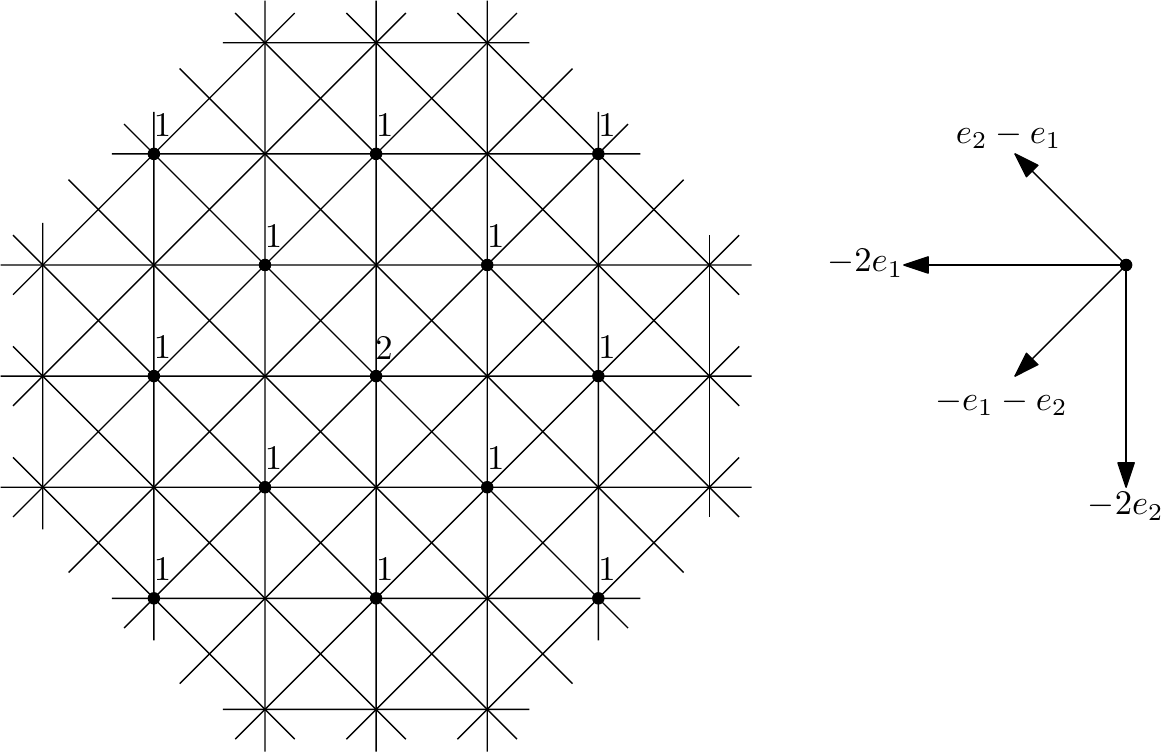}
\end{figure}
the weight multiplicities along the two diagonals of the square do not form log-concave sequences.\footnote{
Note that the Newton polytope of any homogeneous strongly log-concave polynomial is  necessarily a generalized permutohedron of type $A$: Any edge of the Newton polytope should be parallel to $e_i-e_j$ for some $i$ and $j$.}

To strengthen Conjecture \ref{MultiplicityConjecture}, we extend the normalization operator $\mathrm{N}$  to the space of Laurent generating functions
by the formula
\[
\mathrm{N}\Bigg(\sum_{\alpha \in \mathbb{Z}^n} c_\alpha x^\alpha \Bigg)=\sum_{\alpha \in \mathbb{N}^n} c_\alpha \frac{x^\alpha}{\alpha!}.
\]
For $\lambda \in \Lambda$, we introduce the Laurent generating functions 
\[
\mathrm{ch}_{\lambda}(x_1,\ldots,x_m)=\sum_{\mu\in \Lambda} \dim \mathrm{V}(\lambda)_\mu \hspace{0.5mm} x^{\mu-\lambda} \ \ \text{and} \ \  
\underline{\mathrm{ch}}_{\lambda}(x_1,\ldots,x_m)=\sum_{\mu\in \Lambda} \dim \mathrm{M}(\lambda)_\mu \hspace{0.5mm} x^{\mu-\lambda}. 
\]
Note that every monomial 
appearing in the shifted characters $\mathrm{ch}_\lambda$ and $\underline{\mathrm{ch}}_\lambda$
is a product of degree zero monomials of the form $x_ix_j^{-1}$ with $i>j$.

We  tested the following statement for $\lambda=-w \rho-\rho$ and $\delta=(1,\ldots,1)$, for all permutations $w$ in $\mathcal{S}_m$ for $m\leq 6$.\footnote{We point to \url{https://github.com/avstdi/Lorentzian-Polynomials} for code supporting the computations in Section \ref{SectionConjecture}.}

\begin{conjecture}\label{ConjectureIrrep}
The polynomial $\mathrm{N}(x^\delta \hspace{0.3mm} \mathrm{ch}_\lambda(x_1,\ldots,x_m))$ is Lorentzian for any $\lambda \in \Lambda$ and  $\delta \in \mathbb{N}^m$.
\end{conjecture}

For example, when $m=4$ and  $\lambda=-w \rho-\rho$  for the transposition $w=(1,2)$, we have
\begin{multline*}
\mathrm{N}(x_1x_2x_3x_4 \mathrm{ch}_\lambda(x_1,x_2,x_3,x_4))=\frac{4}{24} x_4^4+\frac{2}{6} x_1 x_4^3+\frac{2}{6} x_2 x_4^3+\frac{4}{6} x_3 x_4^3+\frac{3}{4} x_3^2 x_4^2 \\ +\frac{1}{2}x_1 x_2 x_4^2  +\frac{2}{2} x_1 x_3 x_4^2+\frac{2}{2} x_2 x_3 x_4^2+\frac{1}{6}x_3^3 x_4+\frac{1}{2}x_1 x_3^2 x_4+\frac{1}{2}x_2 x_3^2 x_4+\frac{1}{1}x_1 x_2 x_3 x_4,
\end{multline*}
which is a Lorentzian polynomial.
In general, the homogeneous polynomial $\mathrm{N}(x^\delta \mathrm{ch}_\lambda)$ can be computed using the Kazhdan--Lusztig theory \cite[Chapter 8]{Humphreys}.

Theorem \ref{MainTheorem} and Lemma \ref{Translation} show that Conjecture \ref{ConjectureIrrep} holds for any $\delta$ when $\lambda$ is dominant.
We show that Conjecture \ref{ConjectureIrrep}  holds  for any $\delta$ when $\lambda$ is antidominant.

\begin{proposition}\label{VermaLorentzian}
The polynomial $\mathrm{N}(x^\delta \hspace{0.3mm} \underline{\mathrm{ch}}_\lambda(x_1,\ldots,x_m))$ is Lorentzian for any $\lambda \in \Lambda$ and  $\delta \in \mathbb{N}^m$.
\end{proposition}

\begin{proof}
Recall that the dimensions of the weight spaces of $\mathrm{M}(\lambda)$ are given by the Kostant partition function $p$.
In other words, we have
\[
\underline{\mathrm{ch}}_{\lambda}(x_1,\ldots,x_m)=\prod_{i>j} (1+x_ix_j^{-1}+x_i^2x_j^{-2}+\cdots ).
\]
Note that, in the expansion of the above product,\footnote{It is clear that the product is well-defined. Officially, the product occurs in the ring of formal characters of the category $\mathcal{O}$ of $\mathfrak{sl}_m(\mathbb{C})$-modules, denoted $\mathcal{X}$ in \cite[Section 1.15]{Humphreys}.} only the terms of degree at least $-\delta$ contribute to $\mathrm{N}(x^\delta \underline{\mathrm{ch}}_\lambda)$.
Therefore, we may choose a suitably large $\alpha \in \mathbb{N}^m$ depending on $\delta \in \mathbb{N}^m$ so that
\[
\mathrm{N}(x^\delta \underline{\mathrm{ch}}_\lambda)=\mathrm{N}(x^\delta x^{-\beta}\prod_{i>j} (x_j^{\alpha_j}+x_i x_j^{\alpha_j-1}+\cdots+x_i^{\alpha_j})),
\ \ \text{where $\beta_i=(m-i)\alpha_i$ for all $i$.}
\]
Observe that the right-hand side is the $\beta$-th partial derivative of the normalized product of $x^\delta$ and
$\sum_k x_i^{\alpha_j-k}x_j^k$,
whose normalization is the Lorentzian polynomial
\[
 \mathrm{N} (x_j^{\alpha_j}+x_ix_j^{\alpha_j-1}+\cdots+x_i^{\alpha_j})=\frac{1}{\alpha_j!}(x_i+x_j)^{\alpha_j}.
\]
The conclusion now follows from \cite[Corollary 6.8]{Lorentzian}.
\end{proof}

Conjecture \ref{MultiplicityConjecture} for $\lambda$ and $\mu$ 
follows from Conjecture \ref{ConjectureIrrep} for $\lambda$ and a sufficiently large $\delta$.
Conjecture \ref{ConjectureIrrep} for $\lambda$ and $\delta$ follows from
Conjecture \ref{ConjectureIrrep} for $\lambda$ and any $\delta'$ larger than $\delta$ componentwise.

%---------------------------------------------------------------------------
\subsection{Schubert polynomials}\label{SectionSchubert}
%---------------------------------------------------------------------------

For $w \in \mathcal{S}_n$ and $\mu \in \mathbb{Z}^n$,  we define the number $K_{w\mu}$ by 
\[
\mathfrak{S}_w(x_1,\ldots,x_n)=\sum_\mu K_{w\mu} x^\mu.
\]
As before, for $\mu \in \mathbb{Z}^n$ and distinct $i,j \in [m]$, we set
\[
\mu(i,j)=\mu+e_i-e_j.
\]
We note that Theorem \ref{Discrete} can be strengthened as follows.

\begin{proposition}
For any $w \in \mathcal{S}_n$ and any $\mu \in \mathbb{N}^n$, we have
\[
K_{w\mu}^2\geq K_{w\mu(i,j)}K_{w\mu(j,i)} \ \ \text{for any $i,j \in [n]$.}
\]
\end{proposition}

\begin{proof}
	By Theorem \ref{SchubertComplement}, the polynomial $\mathfrak{S}_w^\vee$ is Lorentzian. 
	The inequality follows from \cite[Proposition 9.4]{Lorentzian} applied to the  Lorentzian polynomial $\mathfrak{S}_w^\vee$.
\end{proof}

Are normalized Schubert polynomials Lorentzian?
We  tested the following statement for all permutations in $\mathcal{S}_n$ for $n\leq 8$.

\begin{conjecture}\label{ConjectureSchubert}
The polynomial $\mathrm{N}(\mathfrak{S}_w(x_1,\ldots,x_n))$ is Lorentzian for any $w \in \mathcal{S}_n$.
\end{conjecture}

More generally, we conjecture that, for double Schubert polynomials  \cite[Section 15.5]{MillerSturmfels},
\[
\mathrm{N}(\mathfrak{S}_w(x_1,\ldots,x_n,-y_1,\ldots,-y_n)) \ \ \text{is Lorentzian for any $w \in \mathcal{S}_n$.}
\] 
This would imply that the support of any double Schubert polynomial is $\mathrm{M}$-convex, and hence ``saturated'' \cite[Conjecture 5.2]{MTY}.

\begin{proposition}\label{SchubertSupport}
The support of $\mathfrak{S}_w(x_1,\ldots,x_n)$ is $\mathrm{M}$-convex for any $w \in \mathcal{S}_n$.
\end{proposition}

Proposition \ref{SchubertSupport} was conjectured in \cite[Conjecture 5.1]{MTY}
and  proved in \cite{FMS} using an explicit description of flagged Schur modules.
Here we give an alternative proof based on Theorem \ref{SchubertComplement}.
A similar argument can be used more generally to show that the supports of single quiver polynomials appearing in \cite[Section 17.4]{MillerSturmfels} are $\mathrm{M}$-convex.

\begin{proof}
By Theorem \ref{SchubertComplement},  the support of $\mathfrak{S}_w^\vee$ is $\mathrm{M}$-convex.
It is straightforward to check using the definition of $\mathrm{M}$-convexity the general fact that, if the support of $h(x_1,\ldots,x_n)$ is $\mathrm{M}$-convex, then the support of $x^\mu h(x_1^{-1},\ldots,x_n^{-1})$ is $\mathrm{M}$-convex for any monomial $x^\mu$ divisible by all monomials in the support of $h$.\footnote{The general fact extends matroid duality \cite[Chapter 2]{Oxley}, which is the special case $\mu=(1,\ldots,1)$.}
\end{proof}

\begin{proposition}\label{PropositionAvoidance}
%The polynomial % $\mathfrak{S}_w(x_1,\ldots,x_n)$ and  
%$\mathrm{N}(\mathfrak{S}_w(x_1,\ldots,x_n))$ is Lorentzian
Conjecture \ref{ConjectureSchubert}  holds
when  $w\in \mathcal{S}_n$ avoids the patterns $1423$ and $1432$.
\end{proposition}

%Proposition \ref{PropositionAvoidance} proves Conjecture \ref{ConjectureSchubert} in a special case.

\begin{proof}[Sketch of Proof]
%	The \textit{generating function} $f_J$ of a subset $J\subseteq \mathbb{N}^n$ is the multivariate polynomial
%	\[f_J=\sum_{\alpha\in J}\frac{x^\alpha}{\alpha!}. \]
By \cite[Corollary 6.7]{Lorentzian}, the Lorentzian property of $\mathfrak{S}_w$ implies that of  $\mathrm{N}(\mathfrak{S}_w)$. 
We deduce the Lorentzian property of $\mathfrak{S}_w$ from known results on Schubert and Lorentzian polynomials, for permutations avoiding $1423$ and $1432$.

It is shown in \cite[Theorem 7]{FMS} that, for any $w \in \mathcal{S}_n$, the support of $\mathfrak{S}_w$ is the set of integral points in the Minkowski sum of $n$ matroid polytopes. %in $\mathbb{R}^n$. 
The set $\mathrm{J}_w$  of integral points in the Cartesian product of these matroid polytopes is an $\mathrm{M}$-convex subset of $\mathbb{N}^{n\times n}$, %so its set of integral points $\mathrm{J}_w$ is $\mathrm{M}$-convex. 
and hence the generating function $f_{w}$ of $\mathrm{J}_w$ is a Lorentzian polynomial in $n^2$ variables $x_{ij}$  \cite[Theorem 7.1]{Lorentzian}. 
Since any nonnegative linear change of coordinates preserves the Lorentzian property \cite[Theorem 2.10]{Lorentzian}, substituting the variables $x_{ij}$ by $x_i$ in the generating function $f_w$ gives a Lorentzian polynomial. 
According to  \cite[Corollary 5.6]{FMS2} and \cite[Theorem 1.1]{FanGuo},
this specialization of $f_{w}$ coincides with $\mathfrak{S}_w$ when $w$ avoids $1423$ and $1432$,
and thus $\mathfrak{S}_w$ is Lorentzian for such permutations.
\end{proof}
%In particular, \cite[Corollary 6.7]{Lorentzian} implies that $N(\mathfrak{S}_w(x_1,\ldots,x_n))$ is Lorentzian whenever $w$ avoids $1423$ and $1432$, supporting Conjecture \ref{ConjectureSchubert}.

%Moreover,  \cite[Corollary 6.7]{Lorentzian} implies that   $N(\mathfrak{S}_w(x_1,\ldots,x_n))$ is Lorentzian when $w$ avoids $1423$ and $1432$. 

We note that the Schubert polynomials $\mathfrak{S}_{1423}$ and $\mathfrak{S}_{1432}$ are not Lorentzian.

%---------------------------------------------------------------------------
\subsection{Degree polynomials}
%---------------------------------------------------------------------------

Let $w<w(i,j)$ be a covering relation in the Bruhat order of  $\mathcal{S}_n$ labelled by the transposition of $i<j$ in $[n]$.
The \emph{Chevalley multiplicity} is the assignment
\[
w<w(i,j) \longmapsto  \sum_{i \le k<j} x_k,
\]
where $x_k$ are independent variables.
The \emph{degree polynomial}  of $w \in \mathcal{S}_n$ is the generating function
\[
\mathfrak{D}_w(x_1,\ldots,x_{n-1})=\sum_\mathrm{C} m_\mathrm{C}(x_1,\ldots,x_{n-1}),
\]
where the sum is over all saturated chains $\mathrm{C}$ from the identity permutation to $w$, and $m_\mathrm{C}$ is the product of Chevalley multiplicities of the covering relations in $\mathrm{C}$.
The degree polynomials were introduced by Bernstein, Gelfand, and Gelfand \cite{BGG} and studied from a combinatorial perspective by Postnikov and Stanley \cite{degreepolynomials}.

\begin{proposition}
The degree polynomial $\mathfrak{D}_w(x_1,\ldots,x_{n-1})$ is Lorentzian for any $w \in \mathcal{S}_n$.
\end{proposition}

\begin{proof}
Let $B$ be the group of  upper triangular matrices in $\mathrm{GL}_n(\mathbb{C})$, and let $X_w$ be the closure of the $B$-orbit of the permutation matrix corresponding to $w$ in the flag variety $\mathrm{GL}_n(\mathbb{C})/B$.
By \cite[Proposition 4.2]{degreepolynomials}, the degree polynomial of $w$ is, up to a normalizing constant, the volume polynomial of $X_w$ with respect to the  line bundles associated to the fundamental weights $\varpi_1,\ldots,\varpi_{n-1}$.
The conclusion follows from \cite[Theorem 10.1]{Lorentzian}.
\end{proof}

The same argument shows that the analogous statement holds for Weyl groups in other types.

%---------------------------------------------------------------------------
\subsection{Skew Schur polynomials}
%---------------------------------------------------------------------------
Let $\lambda/\nu$ be a skew Young diagram. 
The \emph{skew Schur polynomial} of $\lambda /\nu$ in $m$ variables is the generating function
\[
s_{\lambda/\nu}(x_1,\ldots,x_m) = \sum_{\mathrm{T}} \ x^{\mu(\mathrm{T})}, \quad x^{\mu(\mathrm{T})}=x_1^{\mu_1(\mathrm{T})}\cdots x_m^{\mu_m(\mathrm{T})}, 
\]
where the sum is over all Young tableaux $\mathrm{T}$ of skew shape $\lambda/\nu$ with entries from $[m]$,
and 
\[
\mu_i(\mathrm{T})=\text{the number of $i$'s among the entries of $\mathrm{T}$}, \ \  \text{for $i=1,\ldots,m$.}
\]
Are normalized skew Schur polynomials Lorentzian?
We  tested the following statement for all partitions $\lambda$ with at most $12$ boxes and at most $6$ parts.

\begin{conjecture}\label{SkewConjecture}
The polynomial $\mathrm{N}(s_{\lambda/\nu}(x_1,\ldots,x_m))$ is Lorentzian for any $\lambda/\nu$.
\end{conjecture}

Theorem \ref{MainTheorem} shows that Conjecture \ref{SkewConjecture} holds when $\nu$ is zero, and Corollary \ref{CorollaryProduct} provides some further evidence.
We remark that the $\mathrm{M}$-convexity of the support of any skew Schur polynomial can be deduced from \cite[Proposition 2.9]{MTY}. 
 
%---------------------------------------------------------------------------
\subsection{Schur $P$-polynomials}
%---------------------------------------------------------------------------

Let $\lambda$ be a \emph{strict partition}, that is, a decreasing sequence of positive integers. 
The \emph{Schur $P$-polynomial} of $\lambda$ in $m$ variables is the generating function
\[
P_\lambda(x_1,\ldots,x_m)=\sum_\mathrm{T} \ x^{\mu(\mathrm{T})}, \quad x^{\mu(\mathrm{T})}=x_1^{\mu_1(\mathrm{T})}\cdots x_m^{\mu_m(\mathrm{T})}, 
\]
where the sum is over all marked shifted Young tableaux of shape $\lambda$ with entries from $[m]$.
See \cite[Chapter III]{Macdonald} for this and other equivalent definitions of the polynomial $P_\lambda$. 

Are normalized Schur $P$-polynomials Lorentzian?
We tested the following statement for all strict partitions $\lambda$ with $\lambda_1\leq 12$ and  at most $4$ parts.

\begin{conjecture}
The polynomial $\mathrm{N}(P_\lambda(x_1,\ldots,x_m))$ is Lorentzian for any strict partition $\lambda$.
\end{conjecture}

The $\mathrm{M}$-convexity of the support of $P_{\lambda}$ was observed in \cite[Proposition 3.5]{MTY}. 

%---------------------------------------------------------------------------
\subsection{Grothendieck polynomials}
%---------------------------------------------------------------------------

\emph{Grothendieck polynomials} are polynomial representatives of the Schubert classes in the Grothendieck ring introduced by Lascoux and Sch\"utzenberger \cite{LS2}.
 If $w$ is the longest permutation  $w_\circ \in \mathcal{S}_n$, then the Grothendieck polynomial of $w$ is the monomial
\[
\mathfrak{G}_{w_\circ}(x_1,\ldots,x_n)=x_1^{n-1}x_2^{n-2} \cdots  x_{n-1}^1.
\]
In general, if $w(i)>w(i+1)$ for some $i$ and $s_i$ is the adjacent transposition $(i \ i+1)$, then
\[
\mathfrak{G}_{ws_i}(x_1,\ldots,x_n)=\pi_i \mathfrak{G}_{w}(x_1,\ldots,x_n), \ \ \text{where $\pi_i=\partial_i- \partial_i x_{i+1}$.}
\]
Let $\ell(w)$ be the degree of the Schubert polynomial of $w$, 
let $d(w)$ be the degree of the Grothendieck polynomial of $w$,
and let $\mathfrak{G}_w^k$ be the degree $\ell(w)+k$ homogeneous component of the Grothendieck polynomial.

\begin{conjecture}\label{GrothendieckConjectureI}
The polynomial $(-1)^{k} \mathrm{N} (\mathfrak{G}^k_w(x_1,\ldots,x_n))$ is Lorentzian for any  $w \in \mathcal{S}_n$ and $k \in \mathbb{N}$.
\end{conjecture}

The $\mathrm{M}$-convexity of the support of $\mathfrak{G}^k_w$ was conjectured in \cite[Conjecture 5.1]{AK} and proved in \cite{EY} when $w$ is a Grassmannian permutation. 
Conjecture \ref{GrothendieckConjectureI} implies Conjecture \ref{ConjectureSchubert} because the degree $\ell(w)$ homogeneous component of $\mathfrak{G}_w$ is the Schubert polynomial $\mathfrak{S}_w$.

We may strengthen Conjecture \ref{GrothendieckConjectureI} in terms of 
 the \emph{homogeneous Grothendieck polynomial} 
\[
\widetilde{\mathfrak{G}}_w(x_1,\ldots,x_n,z)\vcentcolon=\sum_{k=0}^{d(w)-\ell(w)} (-1)^{k} \mathfrak{G}^k_w(x_1,\ldots,x_n) z^{d(w)-\ell(w)-k},
\]
where $z$ is a new variable.
Are normalized homogeneous Grothendieck polynomials Lorentzian?
We tested the following statement for all permutations in $\mathcal{S}_n$ for $n \le 7$.

\begin{conjecture}\label{GrothendieckConjectureII}
The polynomial $\mathrm{N}(\widetilde{\mathfrak{G}}_w(x_1,\ldots,x_n,z))$ is Lorentzian for any $w \in \mathcal{S}_n$.
\end{conjecture}

Conjecture \ref{GrothendieckConjectureII} implies Conjecture \ref{GrothendieckConjectureI} because taking partial derivatives and setting a variable equal to zero preserve the Lorentzian property.
We expect an analogous Lorentzian property for double Grothendieck polynomials.

%---------------------------------------------------------------------------
\subsection{Key polynomials}
%---------------------------------------------------------------------------

\emph{Key polynomials} were introduced by Demazure  for Weyl groups \cite{Demazure} and studied by Lascoux and Sch\"utzenberger for symmetric groups \cite{LS}.
When $\mu \in \mathbb{N}^n$ is a partition, the key polynomial of $\mu$ is the monomial
\[
\kappa_\mu(x_1,\ldots,x_n)=x^\mu=x_1^{\mu_1} \cdots x_n^{\mu_n}.
\]
If  $\mu_i<\mu_{i+1}$ for some $i$ and $s_i$ is the adjacent transposition  $(i \ i+1)$, then
\[
\kappa_\mu(x_1,\ldots,x_n)= \partial_i x_i \kappa_{\nu}, \ \ \text{where $\nu=\mu s_i=(\mu_1,\ldots,\mu_{i+1},\mu_{i},\ldots,\mu_n).$}
\]
We refer to \cite{keypolynomials} for more information about key polynomials.

Are normalized key polynomials Lorentzian?
We tested the following statement for all compositions $\mu$ with at most $12$ boxes  and at most $6$ parts.

\begin{conjecture}\label{KeyConjecture}
The polynomial $\mathrm{N}(\kappa_\mu(x_1,\ldots,x_n))$ is Lorentzian for any $\mu \in \mathbb{N}^n$.
\end{conjecture}

Theorem \ref{MainTheorem} shows that Conjecture \ref{KeyConjecture} holds when $\mu$ is a weakly increasing sequence of nonnegative integers, because in this case the key polynomial of $\mu$ is a Schur polynomial.
The $\mathrm{M}$-convexity of the supports of key polynomials was conjectured in  \cite[Conjecture 3.13]{MTY} and proved in \cite{FMS}.

We remark that key polynomials \cite{Demazure} and Schubert polynomials \cite{KP} are both characters of \emph{flagged Schur modules}.\footnote{Flagged Schur modules are representations of the group of upper triangular matrices in $\mathrm{GL}_n(\mathbb{C})$ labelled by \emph{diagrams}. They are also called flagged dual Weyl modules, and, in special cases, key modules. We refer to  \cite[Section 5]{keypolynomials} and  \cite[Section 4]{Magyar}   for expositions.}
It is shown in \cite[Theorem 11]{FMS} that
the character of any flagged Schur module has $\mathrm{M}$-convex support.
Are normalized characters of flagged Schur modules  Lorentzian?

%%%%%%%%%%%%%%%%%%%%%%%%%%%%%%%%%%%%%%%%%%%%%%%%%%%%%%%%%%%%%%%%%%%%%%%%%%%

\end{document}